\documentclass{amsart}
\usepackage{xy}\xyoption{all}
\newcommand{\sdot}{{\mbox{\normalsize$\cdot$}}} 

\newcommand{\rank}{\mathop{\rm rank}}

\newcommand{\pperp}{\mathrel{\mbox{$\perp\hspace{-0.6em}\perp$}}}
\DeclareMathOperator{\inn}{in}
\newcommand{\mc}{\mathcal}
\newcommand{\mb}{\mathbf}
\newcommand{\bb}{\mathbb}
\DeclareMathOperator{\Edges}{Edges}

\newtheorem{theorem}{Theorem}
\newtheorem{proposition}[theorem]{Proposition}

\begin{document}
\title[The binomial ideal of the intersection axiom]
{The binomial ideal of the intersection axiom\\ for conditional probabilities}
\author[Alex Fink]{Alex Fink$^1$}
\thanks{$^1$ Department of Mathematics, University of California, Berkeley,
{\tt finka\char64math.berkeley.edu}}
\maketitle


\begin{abstract}
The binomial ideal associated with the intersection axiom of  
conditional probability is shown
to be radical and is expressed as an intersection of toric prime ideals.  
This solves a problem
in algebraic statistics posed by Cartwright and Engstr\"om.
\end{abstract}

Conditional independence contraints are a family of
natural constraints on probability distributions,
describing situations in which two random variables 
are independently distributed given knowledge of a third.  
Statistical models built around considerations of conditional independence, 
in particular {\em graphical models} in which the constraints
are encoded in a graph on the random variables, enjoy wide
applicability in determining relationships among 
random variables in statistics and in
dealing with uncertainty in artificial intelligence.

One can take a purely combinatorial perspective on the study of
conditional independence, as does Studen\'y~\cite{Stud},
conceiving of it as a relation on triples of subsets of a set of observables 
which must satisfy certain axioms.  A number of elementary implications 
among conditional independence statements are recognised as axioms.  
Among these are the {\em semi-graphoid axioms},
which are implications of conditional independence statements 
lacking further hypotheses, and hence are purely combinatorial statements.  
The {\em intersection axiom} is also often added to the collection, 
but unlike the semi-graphoid axioms it is not uniformly true;
it is our subject here.

Formally, a conditional independence model $\mc M$ 
is a set of probability distributions
characterised by satisfying several conditional independence constraints.
We will work in the discrete setting, 
where a probability distribution $p$ is a multi-way table of probabilities,
and we follow the notational conventions in~\cite{LoAS}.  

Consider the discrete conditional independence model $\mc M$ given by
$$\{X_1\pperp X_2\mid X_3, X_1\pperp X_3\mid X_2\}$$
where $X_i$ is a random variable taking values in the set $[r_i]=\{1,\ldots,r_i\}$.  
Throughout we assume $r_1\geq2$.
Let $p_{ijk}$ be the unknown probability $P(X_1=i,X_2=j,X_3=k)$ in a distribution
from the model $\mc M$.
The set of distributions in the model~$\mc M$ is the variety
whose defining ideal $I_{\mc M}\subseteq S = \bb C[p_{ijk}]$ is 
\begin{align*}
I_{\mc M}&=(p_{ijk}p_{i'j'k}-p_{ij'k}p_{i'jk} : i,i'\in[r_1], j,j'\in[r_2], k\in[r_3])
\\&\quad{}+(p_{ijk}p_{i'jk'}-p_{ijk'}p_{i'jk} : i,i'\in[r_1], j\in[r_2], k,k'\in[r_3]).
\end{align*}
The intersection axiom is the axiom whose premises are the statements of~$\mc M$
and whose conclusion is~$X_1\pperp(X_2,X_3)$.  This implication
requires the further hypothesis that the distribution~$p$ is in the interior 
of the probability simplex, i.e.\ that no individual probability $p_{ijk}$ is zero.  
It is thus a natural question to ask what can be inferred about
distributions $p$ which may lie on the boundary of the probability simplex. 
In algebraic terms, we are asking for a primary decomposition of~$I_{\mc M}$. 

Our Proposition~\ref{th:set-th} resolves a problem posed by Dustin Cartwright 
and Alexander Engstr\"om in~\cite[p.~152]{LoAS}.
The problem concerned the primary decomposition of $I_{\mc M}$;
they conjectured a description in terms of subgraphs of a complete bipartite graph,
which we show here to be correct.

In the course of this project the author carried out computations of 
primary decompositions for the ideal~$\mc M_I$ for various values
of $r_1$, $r_2$, and~$r_3$ with the computer algebra system 
Singular~\cite{GPS01,GPSprimdec}.
Thomas Kahle has recently written dedicated Macaulay2 code~\cite{Macaulay2}
for binomial primary decompositions~\cite{KahleBPD}, 
in which the same computations may be carried out.

A broad generalisation of this paper's results to the class
of {\em binomial edge ideals} of graphs has been obtained 
by Herzog, Hibi, Hreinsd\'ottir, Kahle, and Rauh \cite{HHHKR}.


Let $K_{p,q}$ be the complete bipartite graph with bipartitioned vertex set
$[p]\amalg [q]$.  We say that a subgraph~$G$ of~$K_{r_2,r_3}$ 
is {\em admissible} if $G$ has vertex set
$[r_2]\amalg[r_3]$ and all connected components of~$G$ are isomorphic to some 
complete bipartite graph $K_{p,q}$ with $p, q \geq 1$.   

Given a subgraph $G$ with edge set $\Edges(G)$, the prime $P_G$ to which it corresponds 
is defined to be 
\begin{equation}\label{eq:P0+P1}
P_G = P^{(0)}_G + P^{(1)}_G
\end{equation}
where
\begin{align*}
P^{(0)}_G &= (p_{ijk} : i\in [r_1],(j,k)\not\in \Edges(G)),\\[2pt]
P^{(1)}_G &= (p_{ijk}p_{i'j'k'}-p_{ij'k'}p_{i'jk} : 
i,i'\in [r_1], \\&\qquad 
\mbox{$j,j'\in [r_2]$ and $k,k'\in [r_3]$ in the same connected component of~$G$}).
\end{align*}
Note that $j$ and~$j'$, and $k$ and~$k'$, need not be distinct.
That is, for $(p_{ijk})$ on the variety $V(P_G)$, $p_{ijk}=0$ for $(j,k)\not\in \Edges(G)$,
and any pair of vectors $p_{\sdot jk}$ and $p_{\sdot j'k'}$ are proportional for 
$(j,k)$ and~$(j',k')$ two edges in $\Edges(G)$ 
in the same connected component of~$G$.  
Later we will also want to refer to the individual summands 
$P^{(1)}_C$ of~$P^{(1)}_G$, where $P^{(1)}_C$ includes only the generators 
$\{p_{ijk}:(j,k)\in C\}$ arising from edges in the connected component~$C$.   

\begin{proposition}\label{th:set-th}\label{conj:Cart-Eng}
The set of minimal primes of the ideal $I_{\mc M}$ is 
$$\{P_G : \mbox{$G$ an admissible graph on~$[r_2]\amalg[r_3]$}\}.$$
\end{proposition}

In particular, the value of $r_1$ is irrelevant to the 
combinatorial nature of the primary decomposition.

Proposition~\ref{th:set-th} was the original conjecture of Cartwright and Engstr\"om.  
It is a purely set-theoretic assertion, and is equivalent to the fact that
\begin{equation}\label{eq:variety primdec}
V(I_{\mc M})=\bigcup_G V(P_G)
\end{equation}
as sets, where the union is over admissible graphs $G$.
The ideas of a proof of Proposition~\ref{th:set-th}
were anticipated in part~4 of the problem stated in~\cite[\S6.6]{LoAS}
which was framed for the prime corresponding to the subgraph $G$,
the case where the conclusion of the intersection axiom is valid;
they extend without great difficulty to the general case.  

We will prove a stronger ideal-theoretic result.  
Let $\prec_{\rm dp}$ be the revlex term order on~$S$
over the lexicographic variable order on subscripts, 
with earlier subscripts more significant:
thus under $\prec_{\rm dp}$, we have 
$p_{111}\prec_{\rm dp}p_{112}\prec_{\rm dp}p_{211}$.

\begin{theorem}\label{th:big}\label{prop:radicality}
The primary decomposition 
\begin{equation}\label{eq:ideal primdec}
I_{\mc M}=\bigcap_G P_G
\end{equation}
holds and is an irredundant decomposition, 
where the union is over admissible graphs $G$ on~$[r_2]\amalg[r_3]$.
We moreover have
$$\inn_{\prec_{\rm dp}} I_{\mc M}=
\inn_{\prec_{\rm dp}}\bigcap_G P_G=
\bigcap_{G}\inn_{\prec_{\rm dp}} P_G.$$
Furthermore, each primary component $\inn_{\prec_{\rm dp}} P_G$ is squarefree, 
so $\inn_{\prec_{\rm dp}} I_{\mc M}$ and hence $I_{\mc M}$ are radical ideals.  
\end{theorem}


It is noted in~\cite[\S6.6]{LoAS} that the number $\eta(p,q)$ of 
admissible graphs~$G$ on $[p]\amalg[q]$ is given by the generating function
\begin{equation}\label{eq:admissible gf}
\exp((e^x-1)(e^y-1)) = \sum_{p,q\geq0} \eta(p,q)\frac{x^py^q}{p!q!}.
\end{equation}
which in that reference is said to follow from manipulations of Stirling numbers.
This equation \eqref{eq:admissible gf} can also be obtained
as a direct consequence of a bivariate form of the exponential
formula for exponential generating functions~\cite[\S5.1]{Stanv2},
using the observation that 
$$(e^x-1)(e^y-1) = \sum_{p,q\geq1} \frac{x^py^q}{p!q!}$$ 
is the exponential generating function for complete bipartite graphs with $p,q\geq1$, 
and these are the possible connected components of admissible graphs.

We now review some standard facts on binomial and toric ideals~\cite{ES}.
Let $I$ be a binomial ideal in $\bb C[x_1,\ldots,x_n]$, generated
by binomials of the form $x^v-x^w$ with $v,w\in\bb N^n$.
There is a lattice $L_I\subseteq Z^n$ such that the localisation
$I_{x_1\cdots x_n}\subseteq\bb C[x_1^{\pm1},\ldots,x_n^{\pm1}]$
has the form $(x^v-1 : v\in L_I)$,
provided that this localisation is a proper ideal, 
i.e.\ $I$ contains no monomial.
If $\phi_I:\bb Z^n\to\bb Z^m$ is a
$\bb Z$\/-linear map whose kernel is $L_I$,  
then $\phi_I$ provides a multigrading with respect to which $I$ is homogeneous.
In statistical terms $\phi_I$ computes the {\em minimal sufficient statistics} 
for the statistical model associated to~$I$.

Given a multivariate Laurent polynomial $f\in\bb C[x_1^{\pm1},\ldots,x_n^{\pm1}]$, 
$f$ lies in~$I_{x_1\cdots x_n}$ if and only if, for each fiber
$F$ of~$\phi_I$, the sum of the coefficients
on all monomials $x^v$ with $v\in F$ is zero.  
With respect to $\bb C[x_1,\ldots,x_n]$
a modified statement holds, as follows.  For each fiber $F$,
consider the graph $\Gamma_F(I)$ whose vertices are the set of vectors in~$F$ 
with all entries nonnegative, and whose edge set is $\{(v,w) : x^v-x^w$ 
is a monomial multiple of a generator of~$I\}$.
In the statistical context these edges are known as {\em moves}.
Then $f$ lies in~$I\subseteq\bb C[x_1,\ldots,x_n]$ if and only
if, for each connected component $C$ of each $\Gamma_F(I)$, the sum of the coefficients
on all monomials $x^v$ with $v\in C$ is zero.
In particular $I$ is determined by this set of connected components.



Viewing $I\subseteq\bb C[x_1^{\pm1},\ldots,x_n^{\pm1}]$ as the
ideal of the toric subvariety of~$(\bb C^\ast)^n$ associated to the lattice polytope~$A$, 
Sturmfels in~\cite{GBTV} shows that the radicals of the 
monomial initial ideals of~$I$ are exactly the Stanley-Reisner ideals 
of regular triangulations of~$A$.  
The {\em Stanley-Reisner ideal} $I_\Delta$ 
of a simplicial complex $\Delta$ on a set~$T$
is the monomial ideal of~$\bb C[x_t : t\in T]$ generated as a vector space by the 
products of variables $x_{t_1}\cdots x_{t_k}$ for which $\{t_1,\ldots,t_k\}$
does not contain a face of~$\Delta$.  Every squarefree monomial ideal is the
Stanley-Reisner ideal of some simplicial complex, and
primary decompositions of Stanley-Reisner
ideals are easily described: $I_\Delta$ is the intersection of the
ideals $(x_t : t\not\in F)$ over all facets $F$ of~$\Delta$. 

Sturmfels also treats explicitly
the ideal $I$ of $2\times 2$ minors of an $r\times s$ matrix $Y=(y_{ij})$, 
of which $P_K := P_{K_{r_2r_3}}$ is a particular case.  In this
case the polytope~$A$ is the product of two simplices, 
$\Delta_{r-1}\times\Delta_{s-1}$.  

\begin{theorem}[\cite{GBTV}]\label{thm:in(C)}
Let $I$ be the ideal of $2\times 2$ minors of an $r\times s$ matrix of indeterminates.
For any term order $\prec$, $\inn_\prec I$ is a squarefree monomial ideal. 
\end{theorem}

This immediately yields the radicality claim of Theorem~\ref{prop:radicality}:
the $\inn_\prec P_G$ are squarefree monomial ideals, so their associated primes
are generated by subsets of the variables $\{p_{ijk}\}$.  

We repeat from~\cite{GBTV} one especially describable example of
an initial ideal of this ideal $I$, namely $\inn_{\prec_{\rm dp}} I$,  
corresponding to the case that $\Delta$ is the so-called staircase triangulation.
Then the vertices of the simplices of $\Delta$ 
correspond to those sets
$\pi$ of entries of the matrix $Y$ which form (``staircase'') paths through~$Y$
starting at the upper-left corner, 
taking only steps right and down, and terminating at the lower left corner. 
Hence to each such $\pi$ corresponds one
primary component $Q_{G,\pi}$, generated by all $(r-1)(s-1)$
indeterminates not lying on~$\pi$.  
Note that staircase paths are maximal subsets of indeterminates
not including both $x_{ij'}$ and $x_{i'j}$ for any $i<i'$ and~$j<j'$.

This framework suffices to understand the primary decomposition
of $\inn_\prec P_G$ for an arbitrary admissible graph~$G$.
Let the connected components of~$G$ be $C_1, \ldots, C_l$,
so that, from~\eqref{eq:P0+P1},
$\inn_\prec P_G$ is the sum of the ideal $\inn_\prec P_G^{(0)}=P_G^{(0)}$ 
and the various ideals $\inn_\prec P^{(1)}_{C_i}$, 
and moreover these summands use disjoint sets of variables.
Suppose that
$\inn_\prec P^{(1)}_{C_i}=\bigcap_j Q_{C_i,j}$ are primary decompositions
of the~$\inn_\prec P^{(1)}_{C_i}$.  Then it follows that we have the primary decomposition
$$\inn_\prec P_G = \bigcap_{\mb j} 
\left(P_G^{(0)} + \sum_{i=1}^l \inn_\prec Q_{C_i,j_i}\right)$$
where $\mb j=(j_1,\ldots,j_l)$ ranges over the Cartesian product
of the index sets in $\bigcap_j Q_{C,j}$.

\begin{proof}[Proof of Theorem~\ref{prop:radicality}]
We begin by proving that the right side of~\eqref{eq:ideal primdec}
is an irredundant primary decomposition.
Let $G$ be an admissible graph.  
For each connected component $C\subseteq G$ and fixed $i$, 
$P^{(1)}_C$ are the determinantal ideal 
of $2\times 2$ minors of the matrix with $r_1$ rows and
columns indexed by~$\Edges(C)$, whose $i,(j,k)$ entry is~$p_{ijk}$. 
Being a determinantal ideal, $P^{(1)}_C$ is prime.
The ideal $P^{(0)}_G$ is also prime, as it is generated
by a collection of variables.
Now $P_G$ is the sum of the prime ideals $P^{(0)}_G$ and
$P^{(1)}_C$ for each $C$, and the generators of these primes involve
pairwise disjoint subsets of the unknowns $p_{ijk}$.  
It follows that $P_G$ itself is prime.  

Irredundance is the assertion that for $G$ and~$G'$ distinct admissible graphs, 
$P_G$ is not contained in~$P_{G'}$.  
As above, we will think of the 3-tensor $(p_{ijk})$ 
as a size $r_2\times r_3$ table whose entries are
vectors $(p_{\sdot jk})$ of length~$r_1$.
Then if $(p_{ijk})\in V(P_G)$,
all nonzero vectors in each subtable determined by a connected component of~$G$
are proportional, while vectors outside of any subtable must be the zero vector.
There is an open dense subset $U_G\subseteq V(P_G)$ such that
for $(p_{ijk})\in U_G$, no vector $p_{\sdot jk}$ associated to
a connected component of~$G$ is zero, and no two associated to
distinct components are dependent.

Now, $G$ may differ from $G'$ in two fashions.
If $G$ contains an edge $(j,k)$ that $G'$ doesn't,
the vector $(p_{\sdot jk})$ is zero on~$V(P_{G'})$ but
is nonzero on~$U_G$: hence $V(P_G)\not\subseteq V(P_{G'})$.
If not, $G\subseteq G'$, but two edges $(j,k),(j',k')$ in different components of~$G$ 
must be in the same component of~$G'$, in which case the vectors $(p_{\sdot jk})$
and $(p_{\sdot j'k'})$ are linearly dependent for~$(p_{ijk})\in V(P_{G'})$
but linearly independent on~$U_G$:
hence also $V(P_G)\not\subseteq V(P_{G'})$.
This proves irredundance.

Now we turn to proving~\eqref{eq:ideal primdec}.
Let $\prec$ be $\prec_{\rm dp}$.
Write $I=I_{\mc M}$.  
It is apparent that $I\subseteq P_G$ for each~$G$.
Indeed, given a generator $f$ of~$I$, without loss of generality
$f=p_{ijk}p_{i'j'k}-p_{ij'k}p_{i'jk}$, 
either both edges
$(j,k)$ and $(j',k)$ lie in~$\Edges(G)$, in which case $f$ is 
a generator of~$P^{(1)}_G$, or one of these edges is not in~$\Edges(G)$,
in which case $f\in P^{(0)}_G$.  Therefore the containments 
\begin{equation*}\label{eq:P decomp}
{\inn}_\prec I \subseteq
\mathop{{\inn}_\prec} \bigcap_G P_G \subseteq
\bigcap_G\mathop{{\inn}_\prec} P_G 
\end{equation*}
hold.
It now suffices to show an equality of Hilbert functions
\begin{equation}\label{eq:H}
H(S/\inn_\prec I) = H(S/\bigcap_G\inn_\prec P_G). 
\end{equation}

In the present case, 
the lattice $L_I$ associated to~$I$ is generated by all vectors of the forms
$e_{ijk}+e_{i'j'k}-e_{ij'k}-e_{i'jk}$ and
$e_{ijk}+e_{i'jk'}-e_{ijk'}-e_{i'jk}$.  
The map $\phi_I:\bb Z^{r_1r_2r_3}\to\bb Z^{r_1+r_2r_3}$ 
sending $(u_{ijk})$ to 
$$\left(\sum_{(j,k)}u_{1jk}, \ldots, \sum_{(j,k)}u_{r_1jk},
\sum_i u_{i11}, \ldots, \sum_i u_{ir_2r_3}\right)$$ 
has kernel $L_I$ and thus
induces the multigrading on~$S$ by minimal sufficient statistics,
with respect to which $I$ is homogeneous.
In fact the analogue of~\eqref{eq:H} using Hilbert functions in the multigrading $\phi$
is also true, and it is this we will prove.

Let $d\in\bb Z^{r_1+r_2r_3}$ be the multidegree of some monomial, and write its components
as $d_i$ for $i\in [r_1]$ and $d_{jk}$ for $j,k\in[r_2]\times[r_3]$.
Let $G(d)$ be the bipartite graph with 
vertex set $[r_2]\amalg[r_3]$ and 
edge set
$\{(j,k) : d_{jk}\neq 0\}$.
We now prove the following two claims:

\noindent{\bf Claim 1}.  $I_d = (P_{G(d)})_d$.

\noindent{\bf Claim 2}.  $(\bigcap_G\inn_\prec P_G)_d = (\inn_\prec P_{G(d)})_d$.

These claims, and the fact that an ideal and its initial ideal have
the same Hilbert function, imply
$$H(\inn_\prec I)(d)=H(I)(d)=H(P_{G(d)})(d)=H(\inn_\prec P_{G(d)})(d)=H(\bigcap_G\inn_\prec P_G)(d),$$
We conclude that \eqref{eq:H} holds, proving Theorem~\ref{prop:radicality}.

\paragraph{\it Proof of Claim~1.} 
Observe first that no polynomial homogeneous of multidegree~$d$
can be divisible by any $p_{ijk}$ with $(j,k)\not\in \Edges(G(d))$.  
Accordingly we have $(P_{G(d)})_d = (P^{(1)}_{G(d)})_d$, in the notation of~\eqref{eq:P0+P1},
and we will work with $P^{(1)}_{G(d)}$ hereafter.

Since $I$ and $P^{(1)}_{G(d)}$ are binomial ideals generated by differences
of monomials, it will suffice to show that the two graphs 
$\Gamma_F(I)$ and $\Gamma_F(P^{(1)}_{G(d)})$ of moves
on the fiber $F=\phi_I^{-1}(d)$ 
have the same partition into connected components.  
The refinement in one direction is clear: $\Gamma_F(I)$ is a subgraph of $\Gamma_F(P^{(1)}_{G(d)})$,
since $I_d\subseteq (P_{G(d)})_d = (P^{(1)}_{G(d)})_d$, and indeed
each generator of $I$ of multidegree at most $d$
is a monomial multiple of a generator of~$P^{(1)}_{G(d)}$. 

So given an edge of~$\Gamma_F(P^{(1)}_{G(d)})$, we must show that this edge is
contained in a connected component of $\Gamma_F(I)$.
Let $u,u'\in F$ be the endpoints of an edge of $\Gamma_F(P^{(1)}_{G(d)})$.
Then $u = u' +e_{ijk}+e_{i'j'k'}-e_{ij'k'}-e_{i'jk}$ for some $i,i'\in[r_1]$
and $(j,k),(j',k')$ edges of $G(d)$ in the same component.  By connectedness, 
there is a path of edges $e_0=(j',k'),e_1,\ldots,e_l=(j,k)$ of~$G(d)$ 
such that $e_i$ and~$e_{i+1}$ share a vertex
for each~$i$.  
Corresponding to this path there exists a sequence of moves 
$(M_m)_{m=0,\ldots,l-1}$ in~$I$, say $M_m = p^{u_m}-p^{u_{m+1}}$,
where $u_0=u'$, $u_l=u$, and where $M_m$ is a monomial multiple of  
$$p_{i,e_m}p_{i_m,e_{m+1}} - p_{i_m,e_m}p_{i,e_{m+1}}$$ 
for some $i_m\in[r_1]$.  So $u$ and~$u'$ are in a single connected
component of~$\Gamma_F(I)$.

\paragraph{\it Proof of Claim~2.} 
Again, one containment is straightforward, namely $\inn_\prec P_{G(d)}\subseteq\bigcap_G \inn_\prec P_G$.
There is an admissible graph $G$ such that $P_G\subseteq P_{G(d)}$.
Such a $G$ can be constructed per the discussion of irredundance,
if we take $p$ to be a generic point of~$V(P_{G(d)})$.  
Then $\inn_\prec P_G\subseteq\inn\prec P_{G(d)}$ and this latter initial
ideal is one of the ideals being intersected in $\bigcap_G \in_\prec P_G$.

For the other containment, 
let $C$ be any connected bipartite graph on vertex set $[r_2]\amalg[r_3]$,  
such that $d_{jk}=0$ for $(j,k)\not\in E(C)$.
By the Stanley-Reisner description of the initial ideal for $\prec_{\rm dp}$,
a monomial $p^u\in S$ of degree $d$ lies in 
$\inn_\prec P_C = \inn_\prec P_{K_{r_2r_3}}$ if and only if $p^u$
is divisible by
$p_{ij'k'}p_{i'jk}$ for some $i<i'$ and~$(j,k)<(j',k')$ lexicographically.


So if $p^u$ is a monomial of multidegree~$d$
lying in $\inn_{\prec_{\rm dp}} P_{G(d)}$, it's divisible by some
$p_{ij'k'}p_{i'jk}$ with $i<i'$ in~$[r_1]$ 
and $(j,k)<(j',k')$ two edges lying in the same connected component of~$G(d)$;
it cannot occur that instead $p^u$ is
divisible by some indeterminate $p_{ij''k''}$ for $(j'',k'')$ not an edge of $G(d)$,
since $p_u$ has multidegree~$d$.  
Now let $G$ be any admissible graph.  If $G(d)$ is not a subset of~$G$,
then $p^u$ is divisible by some indeterminate $p_{ij''k''}$ with $(j'',k'')\not\in E(G)$,
so $p^u\in \inn_\prec P_G$.  Otherwise $G(d)\subseteq G$.  In this case the edges
$(j,k)$ and~$(j',k')$ lie in the same component of~$G$, and
so $p_{ij'k'}p_{i'jk}\mid p^u$ implies $p^u\in\inn_\prec P_G$ again.
Therefore $\inn_\prec P_{G(d)}\supseteq\bigcap_G \inn_\prec P_G$.

\end{proof}

We close with the remark that we can describe explicitly which components
of~\eqref{eq:ideal primdec} contain a given point of $V(I_{\mc M})$. 
Let $p=(p_{ijk})\in\bb C^{r_1r_2r_3}$, and define 
$G(p)$ to be the bipartite graph on~$[r_2]\amalg[r_3]$ with edge set
$\{(j,k) : \mbox{$p_{ijk}\neq 0$ for some $i$}\}$. 
Then the components $V(P_G)$ containing
$(p_{ijk})$ are exactly those for which $G$ can be obtained 
from~$G(p)$ by adding edges which don't unite
two connected components of the latter containing respective 
edges $(j,k)$ and~$(j',k')$ such that $p_{\sdot jk}$ and $p_{\sdot j'k'}$
are not proportional.  If $p\in U_{G(p)}$, 
then these components are exactly those for which $G$ adds only edges which
don't unite two connected components of~$G(p)$, neither of which is an isolated vertex.

\section*{Acknowledgements}
We thank Bernd Sturmfels and 
a referee
for careful readings and for several helpful suggestions.


\begin{thebibliography}{99}
\bibitem{LoAS} M. Drton, B. Sturmfels and S. Sullivant,
{\em Lectures on Algebraic Statistics}, Oberwolfach Seminars vol.~39,
Springer, 2009.
\bibitem{ES} D. Eisenbud and B. Sturmfels, 
Binomial ideals, {\em Duke Math. J.} {\bf 84} (1996), 1--45. 
\bibitem{Macaulay2} D.~R.~Grayson and M.~Stillman, 
{\em Macaulay2, a software system for research in algebraic geometry},
Available at {\tt http://www.math.uiuc.edu/Macaulay2/}.
\bibitem{GPS01} G.-M. Greuel, G. Pfister, H. Sch\"onemann,
{\sc Singular} 3.0 --- A Computer Algebra System for Polynomial Computations.
In M. Kerber and M. Kohlhase: {\em Symbolic Computation and Automated Reasoning,
The Calculemus-2000 Symposium} (2001), 227--233.
\bibitem{GPSprimdec} G.-M. Greuel and G. Pfister,
{\tt primdec.lib}, a {\sc Singular} 3.0 library for computing
the primary decomposition and radical of ideals (2005).
\bibitem{HHHKR} J.~Herzog, T.~Hibi, F.~Hreinsd\'ottir, T.~Kahle, J.~Rauh,
{\em Binomial edge ideals and conditional independence statements},
preprint, {\tt arXiv:0909.4717}.
\bibitem{KahleBPD} T. Kahle, {\tt Binomials.m2}, 
code for binomial primary decomposition in Macaulay2,
{\tt http://personal-homepages.mis.mpg.de/kahle/bpd/index.html}.
\bibitem{GBTV} B. Sturmfels, Gr\"obner bases of toric varieties,
{\em T\=ohoku Math.\ J.}, {\bf 43} (1991), 249--261.
\bibitem{Stanv2} R. P. Stanley, 
{\em Enumerative Combinatorics} vol.~2, Cambridge Studies in 
Advanced Mathematics, vol.~62, Cambridge University Press, Cambridge, 1997.
\bibitem{Stud} M. Studen\'y, {\em Probabilistic Conditional Independence Structures}, 
Information Science and Statistics, Springer-Verlag, New York, 2005. 
\end{thebibliography}
\end{document}